\documentclass[11pt]{article}
\usepackage{amssymb,amsmath,amsthm}
\usepackage{CJK}
\usepackage{amsmath}
\usepackage{geometry}
\usepackage{setspace}
\usepackage{appendix}

\allowdisplaybreaks
\theoremstyle{plain}
\numberwithin{equation}{section}
\newtheorem{Theorem}{Theorem}[section]

\newtheorem{Lemma}{Lemma}[section]

\theoremstyle{definition}

\newcommand\diverg{\mathop{\mbox{\rm div}}}
\newcommand\curl{\mathop{\mbox{\rm curl}}}
\newcommand{\definition}{{\lower .5ex
\hbox{$\>\>\stackrel{\triangle}{=}\>\>$} }}

\begin{document}
 \begin{center}
 {\Large\bf{Optimal time decay estimation for large-solution about 3D compressible MHD equations }
 }
\footnote{$^{\dag}$
Corresponding Author: feichenstudy@163.com (Fei Chen).}
\footnote{$^{**}$
 Email Address: feichenstudy@163.com (Fei Chen); shuai172021@163.com (Shuai Wang); wcb1216@163\\.com (Chuanbao Wang).}
 \footnote{$^{***}$
This work was supported by the National Natural Science Foundation of China [grant number 12101345] and Natural Science Foundation of Shandong Province of China [grant number ZR2021QA017].}

 Shuai Wang$^{*}$, \ Fei Chen$^{*\dag  }$, \ Chuanbao Wang$^{*}$
\\[1ex]
$^* $ School of Mathematics and Statistics, Qingdao University, \\[0.5ex]
Qingdao, Shandong 266071,  China
 \end{center}

 \begin{center}
 \begin{minipage}{15cm}
 \par
 \small  {\bf Abstract:} This paper mainly focus on optimal time decay estimation for large-solution about compressible magnetohydrodynamic equations in 3D whole space, provided that $(\sigma_{0}-1,u_{0},M_{0})\in L^1\cap H^2$. In \cite{Chenyuhui} (Chen et al.,2019), they proved time decay estimation of $\|(\sigma-1,u,M)\|_{H^1}$ being $(1+t)^{-\frac{3}{4}}$. Based on it, we obtained  that of $\|\nabla(\sigma-1,u,M)\|_{H^1}$ being $(1+t)^{-\frac{5}{4}}$ in \cite{Womensa}. Therefore, we are committed to improving that of $\|\nabla^2 (\sigma-1,u,M)\|_{L^2}$ in this paper. Thanks to the method adopted in \cite{Wangwenjun} (Wang and Wen, 2021), we get the optimal time decay estimation to the highest-order derivative for space of solution, which means that time decay estimation of $\|\nabla^2 (\sigma-1,u,M)\|_{L^2}$ is $(1+t)^{-\frac{7}{4}}$.

 \end{minipage}
 \end{center}
  \allowdisplaybreaks
 \vskip2mm
 {\section{Introduction }}
\par
We introduce  the following compressible magnetohydrodynamic (CMHD) equations for $(x,t)\in R^{3}\times R^{+}$
 \begin{eqnarray}\label{1.1}
 \begin{cases}
   \partial_{t}\sigma+\diverg(\sigma u)=0,\\
   \partial_{t}(\sigma u)+\diverg(\sigma u\otimes u)-\mu\Delta u-(\mu+\lambda)\nabla\diverg u+\nabla P-(\curl M)\times M=0,\\
   \partial_{t}M-\nu\Delta M-\curl(u\times M)=0,\quad \diverg M=0,\\
 \end{cases}
\end{eqnarray}
 the functions are density $\sigma\in R^{+}$, velocity field $u\in R^{3}$, magnetic field $M\in R^{3}$ and pressure $P=P(\sigma)=\sigma^\gamma $ (adiabatic exponent $\gamma$ satisfies $\gamma \geq 1$ ). $\mu$ and $\lambda$ with $\mu >0$ and $2\mu+3\lambda>0$ are viscosity coefficients, in addition,  $\nu>0$ expresses magnetic diffusion coefficient. We supplement (\ref{1.1}) by giving the following initial data
  \begin{eqnarray}\label{1.2}
  (\sigma,u,M)(x,0)=(\sigma_{0},u_{0},M_{0})(x)\rightarrow(1,0,0),\quad as~~|x|\rightarrow +\infty.
  \end{eqnarray}
 \par
    About 3 dimensions CMHD equations, above all, we recall some well-posedness results. For classical solutions: global existence under conditions of being vacuum and small energy \cite{Lihailiang}; for large-solutions: local existence \cite{Fanjishan3,Vol'pert}, well-posedness for initial-boundary value question in $H^1$ \cite{Wangdehua}, global existence with $\nu$ needing large enough and $\gamma$ being close to $1$ \cite{Hongguangyi}, as well as global existence and time decay estimations under $L^l(1\leq l<\frac{6}{5})\cap H^3$ \cite{Chenqing} or $L^1\cap H^2$ \cite{Chenyuhui}. For weak solutions, the well-posedness and time decay estimations can be read in \cite{Ducomet,Huxianpeng1,Huxianpeng2,SLiu,Suen1,Suen2,Wu}.
 \par
 Next, we  emphasis  some related paper for time decay estimations about (\ref{1.1}). In regard to small-solutions, time decay estimation on $\mathbb{T}^3$ can be found in Zhu with Zi \cite{Zhulimei}. In the case that  low-frequency parts of initial data are in some $\dot{B}_{2,\infty}^{-b}$ $(b_{1}\geq 2, b_{2}=\frac{2b_{1}}{q}-\frac{b_{1}}{2}, 1-\frac{b_{1}}{2}<b\leq b_{2})$, time decay estimation of $L^\mathfrak{L}$-solution $\big(\mathfrak{L}\in[2,\min(\frac{2b_{1}}{b_{1}-2},4)]\big)$  was studied by Shi with Zhang \cite{Shi}. If initial data is  in $L^l(1\leq l<\frac{6}{5})\cap H^3$, Chen with Tan \cite{Chenqing} proved global existence and showed (\ref{1}) with $f=0,1$; $\alpha=\frac{3}{2}(\frac{1}{l}-\frac{1}{2})+\frac{f}{2}$
 \begin{eqnarray}\label{1}
  \|\nabla^{f}(\sigma-1,u,M)\|_{H^{3-f}}\leq C(1+t)^{-\alpha},
  \end{eqnarray}
  besides, Li with Yu \cite{Lifucai} also got (\ref{1}) in case of $l=1$, as well as Zhang with Zhao \cite{Zhangjianwen} additional proved
  \begin{eqnarray}\label{111}
  \|\partial_{t}(\sigma-1,u,M)\|_{H^1}\leq C(1+t)^{-\frac{3}{2}(\frac{1}{l}-\frac{1}{2})-\frac{1}{2}},
  \end{eqnarray}
   similarly, Pu with Guo \cite{Puxueke} obtained (\ref{1}) to full CMHD equations.  On this basis of (\ref{1}), Gao, Chen with Yao \cite{Gao1} showed (\ref{2}) with $f=2,3$; $\beta_{1}=\frac{3}{2}(\frac{1}{l}-\frac{1}{2})+1$; $\beta_{2}=\frac{3}{2}(\frac{1}{l}-\frac{1}{2})+\frac{f}{2}$
  \begin{eqnarray}\label{2}
  \begin{aligned}
  &\|\nabla^2(\sigma-1,u)\|_{H^1}\leq C(1+t)^{-\beta_{1}},\\
  &\|\nabla^{f} M\|_{H^{3-f}}\leq C(1+t)^{-\beta_{2}},
  \end{aligned}
  \end{eqnarray}
  similarly, Gao, Tao with Yao \cite{Gao2} gained that to full CMHD equations. It's obviously that the decay rates about the higher-order derivatives for space in (\ref{2})  is faster than that in (\ref{1}). Based on Guo with Wang \cite{Guoyan}, Tan with Wang \cite{Tanzhong} established (\ref{3}) which is about the higher-order derivatives for space in the event that initial data is under $H^L (L\geq 3)\cap \dot{H}^n(0\leq n<\frac{3}{2})$
  \begin{eqnarray}\label{3}
  \|\nabla^{f}(\sigma-1,u,M)\|_{H^{L-f}}\leq C(1+t)^{-\frac{f+n}{2}},\quad f=0,1,...,L-1.
  \end{eqnarray}
In addition, in case of $H^S(S\geq 3)\cap \dot{B}^{-s}_{2,\infty}(0\leq s\leq\frac{5}{2})$, Huang, Lin with Wang \cite{Huangwenting} proved time decay estimation about the highest-order derivatives for space
\begin{eqnarray}\label{31}
  \|\nabla^{f}(\sigma-1,u,M)\|_{L^2}\leq C(1+t)^{-\frac{f+s}{2}},\quad 0\leq f\leq S.
  \end{eqnarray}
\par
As for large-solutions,  recently, Chen, Huang with Xu \cite{Chenyuhui} studied global stability and time decay estimation for large-solution in case of $(\sigma_{0}-1,u_{0},M_{0})\in L^1\cap H^2$
\begin{eqnarray}\label{4}
 \|(\sigma-1,u,M)\|_{H^1}
\leq C(1+t)^{-\frac{3}{4}},
  \end{eqnarray}
 based on that, Gao, Wei with Yao \cite{Gao3} gained that about the higher-order derivatives for space of $M$
 \begin{eqnarray}\label{5}
\|\nabla M\|_{H^1}+\| M_t\|_{L^2}\leq C(1+t)^{-\frac{5}{4}},
\end{eqnarray}
and we \cite{Womensa} proved that about the higher-order derivatives for space of $(\sigma,u,M)$
\begin{eqnarray}\label{we}
\|\partial_{t}(\sigma-1,u,M)\|_{L^2}+\|\nabla(\sigma-1,u,M)\|_{H^1}\leq C(1+t)^{-\frac{5}{4}}.
\end{eqnarray}
Hence, in our study, from (\ref{we}), we expect to know:
 whether we can improve time decay estimation of  $\|\nabla^2(\sigma-1,u,M)\|_{L^2}$, which is faster than $(1+t)^{-\frac{5}{4}}$?\\
{\bf Notation:}  Throughout the paper, constant $C>0$  is not about time, which may be variable in different lines. Besides, for purpose of simplifying the writing, denote $\|\cdot\|_{L^2}=\|\cdot\|$, $\|\cdot\|_{H^p}=\|\cdot\|_{p}$ , $\|w_{1}\|_{W}+\|w_{2}\|_{W}=\|(w_{1},w_{2})\|_{W}$ as well as  $\|w_{1}\|_{W}^v+\|w_{2}\|_{W}^v=\|(w_{1},w_{2})\|_{W}^v$.
 \par
  Write Theorem \ref{Theorem1.1} proved in \cite{Chenyuhui} because of the usage more than once in our study.
\begin{Theorem}\label{Theorem1.1}
 Assume conditions are as follows: $\mu>\frac{1}{2}\lambda$, $(\sigma,u,M)$ is a smooth global solution about (\ref{1.1}), $\sigma\in[0,N_{1}]$, $(\sigma_{0},u_{0},M_{0})$ with $0 < a\leq \sigma_{0} $  satisfies admissible conditions, $(\sigma_{0}-1,u_{0},M_{0})\in L^1\cap H^2$, $\sup_{t\geq 0}\|\sigma(\cdot,t)\|_{C^{\beta}}+\sup_{t\geq 0}\|M(\cdot,t)\|_{L^\infty}\leq N_{2}$ with $ 0<\beta<1 $, so, for $\forall\, t>0$, $\exists~~\underline{\sigma}=\underline{\sigma}(a,N_{2})>0$,  it holds
\begin{eqnarray}\label{Th1}
\sigma(x,t)\geq\underline{\sigma},
\end{eqnarray}
\begin{eqnarray}\label{Th2}
\begin{aligned}
\|(\sigma-1,u,M)\|_{2}^2
+\int_{0}^\infty\Big(\|\nabla(\sigma-1)(\tau)\|_{1}^2+\|\nabla (u,M)(\tau)\|_{2}^2 \Big)d\tau \leq C_{0},
\end{aligned}
\end{eqnarray}
\begin{eqnarray}\label{Th3}
\begin{aligned}
\|(\sigma-1,u,M)\|_{1}
\leq C_{0}(1+t)^{-\frac{3}{4}}.
\end{aligned}
\end{eqnarray}
\end{Theorem}
\par
Now, let us give the statement of our conclusion:
\begin{Theorem}\label{Theorem1.2} On the basis of conditions about Theorem 1.1, for $\forall t\geq \widetilde{T}$, $(\sigma,u,M)$ about (\ref{1.1}) has time decay estimation
\begin{eqnarray}\label{Th1.2}
\|\nabla^k(\sigma-1,u,M)\|\leq C(1+t)^{-\frac{3}{4}-\frac{k}{2}},~~k=0,1,2.
 \end{eqnarray}
  Here, the large time $\widetilde{T}$ is going to be defined in Lemma \ref{Lemma2.7}.
 \end{Theorem}
And then, we present the main method and process. In the first step, based on estimations about the second-order derivative for space of $(q,u,M)$ and $\nabla ^2 q$  which were obtained in \cite{Womensa}, we set up the below energy estimation by classical energy method
\begin{eqnarray}\nonumber
\begin{aligned}
&\frac{d}{dt}E(t)+\frac{1}{2}\int_{R^3}\Big(\delta P'(1)|\nabla^2 q|^2+\eta|\nabla^2 \diverg u|^2+\mu|\nabla^3 u|^2+\nu|\nabla^3 M|^2 \Big)dx\\
&\leq C_{1}\delta \Big(\|\nabla ^2 u\|^2+\|\nabla^2 M\|^2\Big),
\end{aligned}
\end{eqnarray}
for a small constant $\delta>0$, a large enough time $T_{1}>0$ and a constant $C_{1}>0$ ( not about time), here
\begin{eqnarray}\nonumber
E(t):=\delta\int_{R^3}\nabla u\cdot\nabla^2 q dx+\frac{p'(1)}{2}\|\nabla^2 q\|^2 +\frac{1}{2}\|\nabla^2 u\|^2+\frac{1}{2}\|\nabla^2 M\|^2.
\end{eqnarray}
Obviously,  $ \int_{R^3}\nabla u\cdot\nabla^2 q dx$ is a difficult term to get time decay estimation. But, owing to the smallness of $\delta$, $E(t)\sim\|\nabla (q,u,M)\|_{1}^2$ , so, in \cite{Womensa}, we only obtain same time decay estimations of $\|\nabla(q,u,M)\|$ and $\|\nabla^2 (q,u,M)\|$. To overcome this difficulty, we adopt the method inspired by Wang with Wen \cite{Wangwenjun}. In \cite{Wangwenjun}, they used a method which has no decay loss of the highest-order derivative for space of solution to prove optimal time decay estimation for small-solution about compressible Navier-Stokes equations with reaction diffusion. Thus, in the second step, the detailed process in Lemma \ref{Lemma2.4}, we establish the estimation of $\int_{R^3}\nabla u\cdot\nabla^2 q^L dx$ ($q^L$ is low-medium frequency portion of $q$ which the details is in Appendix A ), and then take the elimination of it on $E(t)$, it can get
\begin{eqnarray}\nonumber
E(t)-\delta\int_{R^3}\nabla u \cdot \nabla^2 q^L dx\sim\|\nabla^2(q,u,M)\|^2,
\end{eqnarray}
further, for $T_{*}$ ( large enough time) and constant $C_{*}>0$ (not about time), one has
\begin{eqnarray}\nonumber
\|\nabla^2 (q,u,M)(t)\|^2\leq Ce^{-C_{*} t}\|\nabla^2(q,u,M)(T_{*})\|^2+C\int_{T_{*}}^t e^{-C_{*}(t-\tau)}\|\nabla^2(q^L,u^L,M^L)(\tau)\|^2 d\tau.
\end{eqnarray}
Finally, based on time decay estimation for solution about linearized equations, which is the analysis for low-middle-frequency portion (details in Appendix B), time decay estimation of $\|\nabla^2(q^L,u^L,M^L)(\tau)\|^2$ is derived, further, we set up that of $\|\nabla^2(q,u,M)\|^2$.
{\section{Process of proof}}
\par
 The process of proof about Theorem \ref{Theorem1.2} mainly includes four parts: energy estimations, elimination of $\int_{R^3}\nabla u\cdot\nabla^2 q^L dx$, $L_{x}^2 L_{t}^\infty$-norm-estimation for low-middle-frequency portion as well as time decay estimation about nonlinear equations.
Firstly, define $q :=\sigma-1$ and $\eta:=\mu+\lambda$ , then, rewrite (\ref{1.1}) and (\ref{1.2}), one has
\begin{eqnarray}\label{2.1}
 \begin{cases}
   \partial_{t}q+\diverg u=\mathfrak{a},\\
   \partial_{t}u+P'(1)\nabla q-\eta\nabla \diverg u-\mu\Delta u=\mathfrak{b},\\
   \partial_{t}M-\nu\Delta M=\mathfrak{c},
 \end{cases}
\end{eqnarray}
\begin{eqnarray}\label{2.2}
(q,u,M)(x,0)=(q_{0},u_{0},M_{0})(x)\rightarrow (0,0,0),\quad as~|x|\rightarrow +\infty,
\end{eqnarray}
where $\mathfrak{a}$, $\mathfrak{b}$, $\mathfrak{c}$ are defined as
\begin{eqnarray}\label{2.3}
\begin{cases}
\mathfrak{a}:=- u\cdot\nabla q -q \diverg u,\\
\mathfrak{b}:=\frac{1}{q+1}(M\cdot\nabla M+M\cdot \nabla^{t}M)- \frac{q}{q+1}(\mu\Delta u+\eta\nabla \diverg u)-\Big(\frac{P'(1+q)}{1+q}-P'(1)\Big)\nabla q-u\cdot \nabla u,\\
\mathfrak{c}:=-u\cdot \nabla M-(\diverg u)M+M\cdot \nabla u.\\
\end{cases}
\end{eqnarray}
{\subsection {Energy estimations}}
\par
We thanks to the following Lemma \ref{Lemma2.1} and Lemma \ref{Lemma2.2} obtained in \cite{Womensa}, they are estimations about the second-order derivative for space of $(q,u,M)$ as well as $\nabla ^2 q$ respectively, based on this two Lemma, Lemma \ref{Lemma2.3} can be obtained.
\begin{Lemma}\label{Lemma2.1} On the basis of conditions to Theorem \ref{Theorem1.1},  estimation of the second-order derivative for space of solution is as follows
\begin{eqnarray}\label{L2.2}
\begin{aligned}
& \frac{d}{dt}\int_{R^3}\Big(\frac{p'(1)}{2}|\nabla^2 q|^2 +\frac{1}{2}|\nabla^2 u|^2+\frac{1}{2}|\nabla^2 M|^2 \Big)dx+\int_{R^3}\Big(\eta|\nabla^2\diverg u|^2 +\mu|\nabla^3 u|^2+\nu|\nabla^3 M|^2 \Big)dx\\
&\leq C\Big(\|(u,M)\|_{1}+\|(q,u,M,\nabla u)\|^\frac{1}{4}\Big)\Big(\|\nabla^2 q\|^2+\|\nabla^2 (u,M)\|_{1}^2\Big).
\end{aligned}
\end{eqnarray}
\end{Lemma}
\begin{Lemma}\label{Lemma2.2} On the basis of conditions to Theorem \ref{Theorem1.1}, estimation of $\nabla ^2 q$ is  as follows
\begin{eqnarray}\label{L2.3}
\begin{aligned}
&\frac{d}{dt}\int_{R^3}\nabla u\cdot\nabla^2 q dx+\int_{R^3}\frac{7p'(1)}{8}|\nabla^2 q|^2 dx\\
&\leq C\|\nabla^2 u\|_{1}^2 + C\Big(\|(q,u,M)\|_{1}+\|(q,u,M)\|^\frac{1}{4}\Big)\Big(\|\nabla^2q\|^2+\|\nabla^2 M\|_{1}^2\Big).
\end{aligned}
\end{eqnarray}
\end{Lemma}
\par
Next, Lemma \ref{Lemma2.1} and Lemma \ref{Lemma2.2} contribute to Lemma \ref{Lemma2.3}.
\begin{Lemma}\label{Lemma2.3} On the foundation of conditions to Theorem \ref{Theorem1.1}, define $E(t)$
\begin{eqnarray}\label{L2.4.1}
E(t):=\delta\int_{R^3}\nabla u\cdot\nabla^2 q dx+\frac{p'(1)}{2}\|\nabla^2 q\|^2 +\frac{1}{2}\|\nabla^2 u\|^2+\frac{1}{2}\|\nabla^2 M\|^2,
\end{eqnarray}
 for a small constant $\delta>0$, a large enough time $T_{1}>0$ and a constant $C_{1}>0$ (not about time), it establishes
\begin{eqnarray}\label{L2.4.2}
\begin{aligned}
&\frac{d}{dt}E(t)+\frac{1}{2}\int_{R^3}\Big(\delta P'(1)|\nabla^2 q|^2+\eta|\nabla^2 \diverg u|^2+\mu|\nabla^3 u|^2+\nu|\nabla^3 M|^2 \Big)dx\\
&\leq C_{1}\delta \Big(\|\nabla ^2 u\|^2+\|\nabla^2 M\|^2\Big).
\end{aligned}
\end{eqnarray}
\end{Lemma}
\begin{proof} Choosing a small constant $\delta>0$, then, adding up (\ref{L2.2}) with $\delta\times(\ref{L2.3})$ and using  (\ref{Th2}), it may check
\begin{eqnarray}\label{L2.4.3}
\begin{aligned}
&\frac{d}{dt}\int_{R^3}\Big(\delta\nabla u\cdot\nabla^2 q +\frac{p'(1)}{2}|\nabla^2 q|^2 +\frac{1}{2}|\nabla^2 u|^2+\frac{1}{2}|\nabla^2 M|^2\Big)dx\\
&+\frac{3}{4}\int_{R^3}\Big(\delta P'(1)|\nabla^2 q|^2+\eta|\nabla^2 \diverg u|^2+\mu|\nabla^3 u|^2+\nu|\nabla^3 M|^2 \Big)dx\\
&\leq C \Big(\|(q,u,M)\|_{1}+\|(q,u,M,\nabla u)\|^\frac{1}{4}\Big)\Big(\|\nabla^2 q\|^2+\|\nabla^3 u\|^2+\|\nabla^3 M\|^2 \Big)\\
&+C \Big(\|(q,u,M)\|_{1}+\|(q,u,M,\nabla u)\|^\frac{1}{4}\Big)\Big(\|\nabla^2 u\|^2+\|\nabla^2 M\|^2 \Big)\\
&+\widetilde{C}\delta\Big(\|\nabla^2 u\|^2+\|\nabla^2 M\|^2\Big).
\end{aligned}
\end{eqnarray}
Using (\ref{Th3}), it easily check
\begin{eqnarray}\nonumber
\begin{aligned}
\|(q,u,M)\|_{1}+\|(q,u,M,\nabla u)\|^\frac{1}{4}\leq C(1+t)^{-\frac{3}{16}},
\end{aligned}
\end{eqnarray}
so, for a large enough time $T_{1}>0$
\begin{eqnarray}\label{L2.4.4}
\begin{aligned}
C\Big(\|(q,u,M)\|_{1}+\|(q,u,M,\nabla u)\|^\frac{1}{4}\Big)\leq C(1+t)^{-\frac{3}{16}}\leq \frac{1}{4}\min\{\delta P'(1), \mu, \nu,4\widetilde{C}\delta\}.
\end{aligned}
\end{eqnarray}
Thus, (\ref{L2.4.2}) can be proved by combining (\ref{L2.4.3}) and (\ref{L2.4.4})
\begin{eqnarray}\nonumber
\begin{aligned}
&\frac{d}{dt}E(t)+\frac{1}{2}\int_{R^3}\Big(\delta P'(1)|\nabla^2 q|^2+\eta|\nabla^2 \diverg u|^2+\mu|\nabla^3 u|^2+\nu|\nabla^3 M|^2 \Big)dx\\
&\leq C_{1}\delta \Big(\|\nabla ^2 u\|^2+\|\nabla^2 M\|^2 \Big).
\end{aligned}
\end{eqnarray}
\end{proof}
{\subsection {Elimination of $\int_{R^3}\nabla u\cdot\nabla^2 q^L dx$}}
\par
We aim to demonstrate $L_{x}^2 L_{t}^\infty$-norm-estimation of $\nabla^2(q,u,M)$ by the elimination of $\int_{R^3}\nabla u\cdot\nabla^2 q^L dx$ on $E(t)$ (\ref{L2.4.1}).
\begin{Lemma}\label{Lemma2.4} It holds that for $T_{*}$ ( large enough time) 
 and constant $C_{*}>0$ (not about time)
\begin{eqnarray}\nonumber
\|\nabla^2 (q,u,M)(t)\|^2\leq Ce^{-C_{*} t}\|\nabla^2(q,u,M)(T_{*})\|^2+C\int_{T_{*}}^t e^{-C_{*}(t-\tau)}\|\nabla^2(q^L,u^L,M^L)(\tau)\|^2 d\tau.
\end{eqnarray}
\end{Lemma}
\begin{proof} Applying $\nabla(\ref{2.1})_{2}$ by $\nabla^2 q^L$, integrating over $R^3$ and using $(\ref{2.1})_{1}$, one has
\begin{eqnarray}\label{L2.5.1}
\begin{aligned}
\frac{d}{dt}\int_{R^3}\nabla u \cdot \nabla^2 q^L dx&=-P'(1)\int_{R^3}\nabla^2 q\cdot\nabla^2 q^Ldx+\eta\int_{R^3}\nabla^2 \diverg u\cdot\nabla^2 q^L dx+\mu\int_{R^3}\nabla\Delta u\cdot\nabla^2 q^L dx\\&+\int_{R^3}\nabla \mathfrak{b}\cdot\nabla^2 q^Ldx+\int_{R^3}\nabla \diverg u^L\cdot\nabla \diverg u dx-\int_{R^3}\nabla \mathfrak{a}^L \cdot\nabla \diverg u dx.
\end{aligned}
\end{eqnarray}
It's obviously to obtain $(\ref{L2.5.2})$ by H\"{o}lder and Young inequalities
\begin{eqnarray}\label{L2.5.2}
\begin{aligned}
-\frac{d}{dt}\int_{R^3}\nabla u \cdot \nabla^2 q^L dx&\leq
\frac{P'(1)}{4}\|\nabla^2 q\|^2+\frac{\eta}{2}\|\nabla^2 \diverg u\|^2+\frac{\mu}{2}\|\nabla^3 u\|^2+\frac{1}{2}\|\nabla \mathfrak{b}\|^2\\&+\Big(P'(1)+\frac{\mu+\eta+1}{2}\Big)\|\nabla^2 q^L\|^2+\frac{1}{2}\|\nabla \diverg u^L\|^2+\|\nabla \diverg u\|^2+\frac{1}{2}\|\nabla \mathfrak{a}^L\|^2.
\end{aligned}
\end{eqnarray}
According to $(\ref{2.3})_{1}$, by H\"{o}lder and G-N (Gagliardo-Nirenberg) inequalities as well as (\ref{Th2}), we can check
\begin{eqnarray}\label{L2.5.3}
\begin{aligned}
\|\nabla \mathfrak{a}\|&\leq C \|\nabla^2 q\|\|u\|_{L^\infty}+\|\nabla q\|_{L^3}\|\nabla u\|_{L^6}+ \|\nabla q\|_{L^3}\|\diverg u\|_{L^6}+ \| q\|_{L^\infty}\|\nabla\diverg u\| \\
&\leq C(\|u\|_{L^\infty}+\|\nabla q\|_{L^3}+ \| q\|_{L^\infty})(\|\nabla^2 q\|+\|\nabla^2 u\|)\\
&\leq C\Big(\|u\|^{\frac{1}{4}}\|\nabla^2 u\|^{\frac{3}{4}}+\|q\|^{\frac{1}{4}}\|\nabla^2 q\|^{\frac{3}{4}}\Big)(\|\nabla^2 q\|+\|\nabla^2 u\|)\\
&\leq C\|(q,u)\|^{\frac{1}{4}}\|\nabla^2(q,u)\|.
\end{aligned}
\end{eqnarray}
Thanks to $(2.14)$ and $(2.15)$ in \cite{Womensa}, it follows
\begin{eqnarray}\label{L2.5.4}
\begin{aligned}
\|\nabla \mathfrak{b}\|\leq C\Big(\|(u,M)\|_{1}+\|(q,u,M)\|^\frac{1}{4}\Big)(\|\nabla^2 q\|+\|\nabla^2 (u,M)\|_{1}).
\end{aligned}
\end{eqnarray}
By (\ref{A.5}) and Lemma \ref{LA.1}, it gets
\begin{eqnarray}\label{L2.5.5}
\begin{aligned}
\|\nabla \mathfrak{a}^L\|\leq \|\nabla \mathfrak{a}\|+\|\nabla \mathfrak{a}^h\|\leq C \|\nabla \mathfrak{a}\|.
\end{aligned}
\end{eqnarray}
Combining (\ref{L2.5.2})-(\ref{L2.5.5}), we obtain
\begin{eqnarray}\label{L2.5.6}
\begin{aligned}
-\frac{d}{dt}\int_{R^3}\nabla u \cdot \nabla^2 q^L dx&\leq
\frac{P'(1)}{4}\|\nabla^2 q\|^2+\frac{\eta}{2}\|\nabla^2 \diverg u\|^2+\frac{\mu}{2}\|\nabla^3 u\|^2+C\|\nabla^2 q^L\|^2\\&+\frac{1}{2}\|\nabla \diverg u^L\|^2+\|\nabla \diverg u\|^2\\&+C\Big(\|(u,M)\|_{1}^2+\|(q,u,M)\|^\frac{1}{2}\Big)\Big(\|\nabla^2 q\|^2+\|\nabla^2 (u,M)\|_{1}^2 \Big).
\end{aligned}
\end{eqnarray}
Similar (\ref{L2.4.4}), using (\ref{Th3}), for a large enough time $T_{2}>0$
\begin{eqnarray}\label{L2.5.66}
\begin{aligned}
C\Big(\|(u,M)\|_{1}^2+\|(q,u,M)\|^\frac{1}{2}\Big)\leq \frac{P'(1)}{8}.
\end{aligned}
\end{eqnarray}
Summing $\delta\times(\ref{L2.5.6})$ with (\ref{L2.4.2}), using (\ref{L2.5.66}) and then checking that for $T_{*}=\max\{ T_{1},T_{2}\}$
\begin{eqnarray}\label{L2.5.7}
\begin{aligned}
&\frac{d}{dt}\Big(E(t)-\delta\int_{R^3}\nabla u \cdot \nabla^2 q^L dx\Big)+\frac{1}{2}\int_{R^3}\Big(\delta P'(1)|\nabla^2 q|^2+\eta|\nabla^2 \diverg u|^2+\mu|\nabla^3 u|^2+\nu|\nabla^3 M|^2 \Big)dx\\
&\leq\frac{P'(1)}{4}\delta\|\nabla^2 q\|^2+\frac{\eta}{2}\delta\|\nabla^2 \diverg u\|^2+\frac{\mu}{2}\delta\|\nabla^3 u\|^2+C\delta\|\nabla^2 q^L\|^2+\frac{1}{2}\delta\|\nabla \diverg u^L\|^2\\&+\delta\|\nabla \diverg u\|^2+ C_{1}\delta \Big(\|\nabla ^2 u\|^2+\|\nabla^2 M\|^2 \Big)+\frac{P'(1)}{8}\delta\Big(\|\nabla^2 q\|^2+\|\nabla^2 u\|_{1}^2+\|\nabla^2 M\|_{1}^2 \Big).
\end{aligned}
\end{eqnarray}
By Lemma \ref{A.1}, it can check
\begin{eqnarray}\label{L2.5.8}
\begin{aligned}
\frac{\mu}{2}\|\nabla^3 u\|^2+\frac{\nu}{2}\|\nabla^3 M\|^2\geq \frac{\mu}{4}\|\nabla^3 u\|^2+\frac{\mu}{4}B_{0}^2\|\nabla^2 u^h\|^2+\frac{\nu}{4}\|\nabla^3 M\|^2+\frac{\nu}{4}B_{0}^2\|\nabla^2 M^h\|^2.
\end{aligned}
\end{eqnarray}
Putting (\ref{L2.5.8}) into (\ref{L2.5.7}), and then adding $\frac{\mu}{4}B_{0}^2\|\nabla^2 u^L\|^2+\frac{\nu}{4}B_{0}^2\|\nabla^2 M^L\|^2$ on either side of the result, it follows
\begin{eqnarray}\label{L2.5.9}
\begin{aligned}
&\frac{d}{dt}\Big(E(t)-\delta\int_{R^3}\nabla u \cdot \nabla^2 q^L dx\Big)+\frac{P'(1)}{8}\delta\|\nabla^2 q\|^2+\frac{\eta}{2}\|\nabla^2 \diverg u\|^2\\
&+\frac{\mu}{4}\|\nabla^3 u\|^2+\frac{\nu}{4}\|\nabla^3 M\|^2+\frac{\mu}{8}B_{0}^2\|\nabla^2 u\|^2+\frac{\nu}{8}B_{0}^2\|\nabla^2 M\|^2\\
&\leq 
\frac{\eta}{2}\delta\|\nabla^2 \diverg u\|^2+\frac{\mu}{2}\delta\|\nabla^3 u\|^2+C\delta\|\nabla^2 q^L\|^2+\frac{1}{2}\delta\|\nabla \diverg u^L\|^2+\delta\|\nabla \diverg u\|^2\\&+ C_{1}\delta \Big(\|\nabla ^2 u\|^2+\|\nabla^2 M\|^2 \Big)+\frac{P'(1)}{8}\delta \Big(\|\nabla^2 u\|_{1}^2+\|\nabla^2 M\|_{1}^2 \Big)\\
&+\frac{\mu}{4}B_{0}^2\|\nabla^2 u^L\|^2+\frac{\nu}{4}B_{0}^2\|\nabla^2 M^L\|^2,
\end{aligned}
\end{eqnarray}
then, noticing $\delta\leq\min\Big\{\frac{1}{8},\frac{\mu}{2P'(1)},\frac{\nu}{P'(1)}\Big\}:=\delta_{0}$ and then 
$B_{0}^2 \geq \max\Big\{\frac{48\delta_{0}}{\mu},\frac{48C_{1}\delta_{0}}{\mu},\frac{6P'(1)\delta_{0}}{\mu},\frac{32C_{1}\delta_{0}}{\nu},\frac{4P'(1)\delta_{0}}{\nu}\Big\}$,
we can check it out
\begin{eqnarray}\label{L2.5.10}
\begin{aligned}
&\frac{d}{dt}\Big(E(t)-\delta\int_{R^3}\nabla u \cdot \nabla^2 q^L dx\Big)+\frac{P'(1)}{8}\delta\|\nabla^2 q\|^2+\frac{\eta}{4}\|\nabla^2 \diverg u\|^2\\
&+\frac{\mu}{8}\|\nabla^3 u\|^2+\frac{\nu}{8}\|\nabla^3 M\|^2+\frac{\mu}{16}B_{0}^2\|\nabla^2 u\|^2+\frac{\nu}{16}B_{0}^2\|\nabla^2 M\|^2\\
&\leq C\|\nabla^2(q^L,u^L,M^L)\|^2.
\end{aligned}
\end{eqnarray}
By (\ref{A.5}) and Lemma \ref{LA.1} respectively, it proves
\begin{eqnarray}\label{L2.5.11}
\begin{aligned}
&E(t)-\delta\int_{R^3}\nabla u \cdot \nabla^2 q^L dx=\delta\int_{R^3}\nabla u \cdot \nabla^2 q^h dx+\frac{P'(1)}{2}\|\nabla^2 q\|^2+\frac{1}{2}\|\nabla^2 u\|^2+\frac{1}{2}\|\nabla^2 M\|^2,\\
&\delta\int_{R^3}\nabla u \cdot \nabla^2 q^h dx
=-\delta\int_{R^3} \nabla \diverg u \cdot\nabla q^h dx
\leq\frac{\delta}{2}\|\nabla q^h\|^2+\frac{\delta}{2}\|\nabla \diverg u\|^2
\leq\frac{\delta}{2}\|\nabla^2 q\|^2+\frac{\delta}{2}\|\nabla^2 u\|^2,
\end{aligned}
\end{eqnarray}
based on the smallness of $\delta$, it implies
\begin{eqnarray}\label{L2.5.12}
E(t)-\delta\int_{R^3}\nabla u \cdot \nabla^2 q^L dx\sim\|\nabla^2(q,u,M)\|^2.
\end{eqnarray}
Together with $(\ref{L2.5.10})$ and $(\ref{L2.5.12})$, a constant $C_{*}$ can exist so that
\begin{eqnarray}\label{L2.5.13}
\begin{aligned}
&\frac{d}{dt}\Big(E(t)-\delta\int_{R^3}\nabla u \cdot \nabla^2 q^L dx\Big)+C_{*}\Big(E(t)-\delta\int_{R^3}\nabla u \cdot \nabla^2 q^L dx \Big)\\
&\leq C\|\nabla^2(q^L,u^L,M^L)\|^2.
\end{aligned}
\end{eqnarray}
Integrating $(\ref{L2.5.13})\times e^{C_{*}t}$ on $[T_{*},t]$ which respects to time, we can check it out
\begin{eqnarray}\label{L2.5.14}
\begin{aligned}
E(t)-\delta\int_{R^3}\nabla u \cdot \nabla^2 q^L dx
&\leq C e^{-C_{*} t}\Big(E(T_{*})-\delta\int_{R^3}\nabla u(T_{*})\cdot \nabla^2 q^L(T_{*}) dx\Big)\\
&+C\int_{T_{*}}^t e^{-C_{*}(t-\tau)}\|\nabla^2(q^L,u^L,M^L)(\tau)\|^2 d\tau.
\end{aligned}
\end{eqnarray}
The combination of $(\ref{L2.5.12})$ and $(\ref{L2.5.14})$ leads to
\begin{eqnarray}\nonumber
\|\nabla^2 (q,u,M)(t)\|^2\leq Ce^{-C_{*} t}\|\nabla^2(q,u,M)(T_{*})\|^2+C\int_{T_{*}}^t e^{-C_{*}(t-\tau)}\|\nabla^2(q^L,u^L,M^L)(\tau)\|^2 d\tau.
\end{eqnarray}
\end{proof}
{\subsection {$L_{x}^2 L_{t}^\infty$-norm-estimation for low-middle-frequency portion}}
\par
This part is $L_{x}^2 L_{t}^\infty$-norm-estimation for low-middle-frequency portion of $(q,u,M)$ about nonlinear equations $(\ref{2.1})$ and $(\ref{2.2})$, which is based on the analysis of linearized equations in Appendix B.
\par
Define  differential operator $\mathbf{Q}$
\begin{eqnarray}\label{Q}
\mathbf{Q}=
\begin{pmatrix}
0& \diverg &0\\
P'(1)\nabla &-\eta\nabla  \diverg-\mu\Delta &0\\
0&0&-\nu\Delta
\end{pmatrix}
\end{eqnarray}
 so, $(\ref{2.1})$ and $(\ref{2.2})$ can be rewritten as
 \begin{eqnarray}\label{G1}
 \begin{cases}
   \partial_{t}X+\mathbf{Q}X=F(X),\\
   X|_{t=0}=X(0),
 \end{cases}
\end{eqnarray}
where $X(t):=(q(t),u(t),M(t))^{T}$, $F(X):=(\mathfrak{a},\mathfrak{b},\mathfrak{c})^{T}$, and $X(0):=(q_{0},u_{0},M_{0})^{T}$. Further, denote
 $\widetilde{X}(t):=(\widetilde{q}(t),\widetilde{u}(t),\widetilde{M}(t))^{T}$, the linearized equations are as below
\begin{eqnarray}\label{G2}
 \begin{cases}
   \partial_{t}\widetilde{X}+\mathbf{Q}\widetilde{X}=0,\\
   \widetilde{X}|_{t=0}=X(0),
 \end{cases}
\end{eqnarray}
we can calculate its solution: $\widetilde{X}=\mathbf{q}(t)X(0),~~\mathbf{q}(t)=e^{-t\mathbf{Q}}$.
\par
In Appendix B, we analyze the estimations for  low-middle-frequency portion of solution about equations $(\ref{G2})$, which leads to Lemma \ref{Lemma2.5}.
\begin{Lemma}\label{Lemma2.5}  Suppose $1\leq j\leq 2$, the following estimation  holds for  $\forall ~\text{integer}~ k\geq0$
\begin{eqnarray}\label{L2.6}
\|\nabla^k \big(\mathbf{q}(t)X^L(0)\big)\|\leq C\|X(0)\|_{L^j}(1+t)^{-[\frac{3}{2}(\frac{1}{j}-\frac{1}{2})+\frac{k}{2}]}.
\end{eqnarray}
\end{Lemma}
\begin{proof}
By Plancherel theorem, (\ref{B.12}) and (\ref{B.14}) with $\mathfrak{b}=b_{0}$, $\mathfrak{B}=B_{0}$, one has
\begin{eqnarray}\label{L}
\begin{aligned}
\|\partial_{x}^k(\widetilde{q}^L,\widetilde{\Upsilon}^L,\widetilde{M}^L)(t)\|&=\|(i\xi)^k(\widehat{\widetilde{q}^L},\widehat{\widetilde{\Upsilon}^L},\widehat{\widetilde{M}^L})(t)\|_{L_{\xi}^2}\\
&=\bigg(\int_{R^3}|(i\xi)^k(\widehat{\widetilde{q}^L},\widehat{\widetilde{\Upsilon}^L},\widehat{\widetilde{M}^L})(\xi,t)|^2 d\xi\bigg)^{\frac{1}{2}}\\
&\leq C \bigg(\int_{|\xi|\leq B_{0}}|\xi|^{2|k|}|(\widehat{\widetilde{q}},\widehat{\widetilde{\Upsilon}},\widehat{\widetilde{M}})(\xi,t)|^2 d\xi\bigg)^{\frac{1}{2}}\\
&\leq C \bigg(\int_{|\xi|\leq b_{0}}|\xi|^{2|k|}e^{-C_{l}|\xi|^2 t}|(\widehat{q},\widehat{\Upsilon},\widehat{M})(\xi,0)|^2 d\xi\bigg)^{\frac{1}{2}}\\
&+C \bigg(\int_{b_{0}\leq|\xi|\leq B_{0}}|\xi|^{2|k|}e^{-\varsigma t}|(\widehat{q},\widehat{\Upsilon},\widehat{M})(\xi,0)|^2 d\xi\bigg)^{\frac{1}{2}},
\end{aligned}
\end{eqnarray}
then, taking H\"{o}lder as well as Hausdorff-Young inequalities on (\ref{L}),  for $\frac{1}{j}+\frac{1}{j'}=1$, $1\leq j\leq 2\leq j'\leq\infty$, it can check
\begin{eqnarray}\label{LL}
\begin{aligned}
\|\partial_{x}^k(\widetilde{q}^L,\widetilde{\Upsilon}^L,\widetilde{M}^L)(t)\|&\leq C \|(\widehat{q},\widehat{\Upsilon},\widehat{M})(0)\|_{L_{\xi}^{j'}}(1+t)^{-[\frac{3}{2}(\frac{1}{2}-\frac{1}{j'})+\frac{|k|}{2}]}\\
&\leq C \|(q,\Upsilon,M)(0)\|_{L^{j}}(1+t)^{-[\frac{3}{2}(\frac{1}{j}-\frac{1}{2})+\frac{|k|}{2}]}.
\end{aligned}
\end{eqnarray}
In the same way, from (\ref{B.16}), we obtain
\begin{eqnarray}\label{LLL}
\begin{aligned}
\|\partial_{x}^k(\widetilde{\Gamma u})^L(t)\|&\leq C \bigg(\int_{|\xi|\leq B_{0}}|\xi|^{2|k|}|\widehat{\widetilde{\Gamma u}}(\xi,t)|^2 d\xi\bigg)^{\frac{1}{2}}\\
&\leq C \bigg(\int_{|\xi|\leq B_{0}}|\xi|^{2|k|}e^{-\mu|\xi|^2 t}|\widehat{\Gamma u}(\xi,0)|^2 d\xi\bigg)^{\frac{1}{2}}\\
&\leq C \| u_{0}\|_{L^{j}}(1+t)^{-[\frac{3}{2}(\frac{1}{j}-\frac{1}{2})+\frac{|k|}{2}]}.
\end{aligned}
\end{eqnarray}
Combining (\ref{LL}) with (\ref{LLL}), we have verified (\ref{L2.6}).
\end{proof}
 The solution about equations $(\ref{G1})$ can be calculated by the semigroup method and Duhamel principle,
\begin{eqnarray}\label{G3}
X(t)=\mathbf{q}(t)X(0)+\int_{0}^t \mathbf{q}(t-\tau)F(X)(\tau)d\tau.
\end{eqnarray}
\par
Combining Lemma \ref{Lemma2.5} and $(\ref{G3})$, we obtain time decay estimation for low-middle-frequency portion of $(q,u,M)$ about nonlinear equations  $(\ref{2.1})$ and $(\ref{2.2})$.
\begin{Lemma}\label{Lemma2.6}  The following estimation  holds for  $\forall ~\text{integer}~ k\geq0$, $1\leq p\leq 2$
\begin{eqnarray}\label{L2.7}
\begin{aligned}
\|\nabla^k X^L(t)\|&\leq C\|X(0)\|_{L^1}(1+t)^{-\frac{3}{4}-\frac{k}{2}}+C\int_{0}^\frac{t}{2}\|F(X)(\tau)\|_{L^1}(1+t-\tau)^{-\frac{3}{4}-\frac{k}{2}}d\tau\\
&+C\int_{\frac{t}{2}}^t\|F(X)(\tau)\|(1+t-\tau)^{-\frac{k}{2}}d\tau.
\end{aligned}
\end{eqnarray}
\end{Lemma}
{\subsection{Time decay estimation about nonlinear equations}}
\par
We will use Lemma \ref{Lemma2.4} and \ref{Lemma2.6} to arrive at  time decay estimation for $(\sigma,u,M)$ about nonlinear equations (\ref{2.1}) and (\ref{2.2}).
\begin{Lemma}\label{Lemma2.7} On the basis of conditions about Theorem \ref{Theorem1.1}, $\exists~\text{a large constant}~\widetilde{T}>0$ for $\forall t\geq \widetilde{T}$, $(\sigma,u,M)$ about (\ref{1.1}) has time decay estimation
\begin{eqnarray}\label{L2.8}
\|\nabla^k(\sigma-1,u,M)\|\leq C(1+t)^{-\frac{3}{4}-\frac{k}{2}},k=0,1,2.
 \end{eqnarray}
 \end{Lemma}
\begin{proof} Denote
\begin{eqnarray}\label{L2.8.1}
N(t):=\sup\limits_{0\leq\tau\leq t}\sum_{n=0}^2(1+\tau)^{\frac{3}{4}+\frac{n}{2}}\|\nabla^n(q,u,M)(\tau)\|
\end{eqnarray}
so, it has
\begin{eqnarray}\label{L2.8.2}
\|\nabla^n(q,u,M)(\tau)\|\leq CN(t)(1+\tau)^{-\frac{3}{4}-\frac{n}{2}},\quad 0\leq\tau\leq t,\quad 0\leq n\leq 2.
\end{eqnarray}
It's easy to check the following inequalities by (\ref{Th1}), H\"{o}lder and G-N inequalities
\begin{eqnarray}\label{L2.8.3}
\|\mathfrak{a}(\tau)\|_{L^1}\leq C\|(q,u)\|\|\nabla(q,u)\|,
\end{eqnarray}
\begin{eqnarray}\label{L2.8.4}
\begin{aligned}
\|\mathfrak{b}(\tau)\|_{L^1}&\leq C(\|M\|\|\nabla M\|+\|q\|\|\nabla^2 u\|+\|q\|\|\nabla q\|+\|u\|\|\nabla u\|)\\
&\leq C \Big(\|(q,u,M)\|\|\nabla(q,u,M)\|+\|q\|\|\nabla u\|^{\frac{1}{2}}\|\nabla^3 u\|^{\frac{1}{2}}\Big),
\end{aligned}
\end{eqnarray}
and
\begin{eqnarray}\label{L2.8.5}
\|\mathfrak{c}(\tau)\|_{L^1}\leq C\|(u,M)\|\|\nabla(u,M)\|.
\end{eqnarray}
By the combination of (\ref{L2.8.3})-(\ref{L2.8.5}) and the usage of (\ref{Th3}), one has
\begin{eqnarray}\label{L2.8.6}
\begin{aligned}
\|F(X)(\tau)\|_{L^1}&\leq \|\mathfrak{a}(\tau)\|_{L^1}+\|\mathfrak{b}(\tau)\|_{L^1}+\|\mathfrak{c}(\tau)\|_{L^1}\\
&\leq C(1+\tau)^{-\frac{3}{2}}+C \|q\|\|\nabla u\|^{\frac{1}{2}}\|\nabla^3 u\|^{\frac{1}{2}}.
\end{aligned}
\end{eqnarray}
Similarly, and using (\ref{L2.8.2}), (\ref{Th2}) and (\ref{Th3}), it obtains
\begin{eqnarray}\label{L2.8.7}
\|\mathfrak{a}(\tau)\|\leq C\|\nabla^2(q,u)\|\|(q,u)\|_{1}\leq C N(t)(1+\tau)^{-\frac{7}{4}}(1+\tau)^{-\frac{3}{4}}\leq C N(t)(1+\tau)^{-\frac{10}{4}},
\end{eqnarray}
\begin{eqnarray}\label{L2.8.8}
\begin{aligned}
\|\mathfrak{b}(\tau)\|&\leq C(\|M\|_{L^3}\|\nabla M\|_{L^6}+\|q\|_{L^\infty}\|\nabla^2 u\|+\|q\|_{L^3}\|\nabla q\|_{L^6}+\|u\|_{L^3}\|\nabla u\|_{L^6})\\
&\leq C \Big(\|\nabla^2(q,u,M)\|\|(q,u,M)\|_{1}+\|\nabla^2 u\|\|\nabla q\|^{\frac{1}{2}}\|\nabla^2 q\|^{\frac{1}{2}}\Big)\\
&\leq C \Big(N(t)(1+\tau)^{-\frac{7}{4}}(1+\tau)^{-\frac{3}{4}}+N(t)(1+\tau)^{-\frac{7}{4}}(1+\tau)^{-\frac{3}{8}}\Big)\\
&\leq C \Big(N(t)(1+\tau)^{-\frac{10}{4}}+N(t)(1+\tau)^{-\frac{17}{8}}\Big)\\
&\leq CN(t)(1+\tau)^{-\frac{17}{8}},
\end{aligned}
\end{eqnarray}
and
\begin{eqnarray}\label{L2.8.9}
\|\mathfrak{c}(\tau)\|\leq C\|\nabla^2(u,M)\|\|(u,M)\|_{1}\leq C N(t)(1+\tau)^{-\frac{10}{4}}.
\end{eqnarray}
By the combination of (\ref{L2.8.7})-(\ref{L2.8.9}), one has
\begin{eqnarray}\label{L2.8.10}
\begin{aligned}
\|F(X)(\tau)\|&\leq \|\mathfrak{a}(\tau)\|+\|\mathfrak{b}(\tau)\|+\|\mathfrak{c}(\tau)\|\\
&\leq CN(t)(1+\tau)^{-\frac{17}{8}}.
\end{aligned}
\end{eqnarray}
According to (\ref{L2.8.6}), (\ref{L2.8.10}) and Lemma \ref{Lemma2.6}, for $k\in[0,2]$, one may check
\begin{eqnarray}\label{L2.8.11}
\begin{aligned}
\|\nabla^k X^L(t)\|&\leq C\|X(0)\|_{L^1}(1+t)^{-\frac{3}{4}-\frac{k}{2}}\\&+C \int_{0}^\frac{t}{2}\Big((1+\tau)^{-\frac{3}{2}}+ \|q\|\|\nabla u\|^{\frac{1}{2}}\|\nabla^3 u\|^{\frac{1}{2}}\Big)(1+t-\tau)^{-\frac{3}{4}-\frac{k}{2}}d\tau\\
&+CN(t)\int_{\frac{t}{2}}^t(1+\tau)^{-\frac{17}{8}}(1+t-\tau)^{-\frac{k}{2}}d\tau\\
&\leq C\|X(0)\|_{L^1}(1+t)^{-\frac{3}{4}-\frac{k}{2}}+C(1+t)^{-\frac{3}{4}-\frac{k}{2}}+CN(t)(1+t)^{-\frac{9}{8}-\frac{k}{2}}\\
&\leq C (1+t)^{-\frac{3}{4}-\frac{k}{2}}\Big(\|X(0)\|_{L^1}+1+N(t)(1+t)^{-\frac{3}{8}}\Big),
\end{aligned}
\end{eqnarray}
where we have adopted Young inequality, (\ref{Th3}) and (\ref{Th2}) to calculate
\begin{eqnarray}\nonumber
\begin{aligned}
&\int_{0}^\frac{t}{2}\|q\|\|\nabla u\|^{\frac{1}{2}}\|\nabla^3 u\|^{\frac{1}{2}}(1+t-\tau)^{-\frac{3}{4}-\frac{k}{2}} d\tau\\
&\leq C \int_{0}^\frac{t}{2}\Big(\|(q,\nabla u)(\tau)\|^2+\|\nabla^3 u(\tau)\|^2 \Big)(1+t-\tau)^{-\frac{3}{4}-\frac{k}{2}} d\tau\\
&\leq C \int_{0}^\frac{t}{2}\Big((1+\tau)^{-\frac{3}{2}}+\|\nabla^3 u(\tau)\|^2 \Big)(1+t-\tau)^{-\frac{3}{4}-\frac{k}{2}} d\tau\\
&\leq C (1+t)^{-\frac{3}{4}-\frac{k}{2}}.
\end{aligned}
\end{eqnarray}
From Lemma \ref{Lemma2.4}, and use (\ref{L2.8.11}), yields directly for $k\in[0,2]$, $t\geq T_{*}$
\begin{eqnarray}\label{L2.8.12}
\begin{aligned}
\|\nabla^2 X(t)\|^2&\leq C e^{-C_{*} t}\|\nabla^2 X(T_{*})\|^2+C\int_{T_{*}}^t e^{-C_{*}(t-\tau)}(1+\tau)^{-\frac{7}{2}}\Big(\|X(0)\|_{L^1}^2+1+N^2(\tau)(1+\tau)^{-\frac{3}{4}}\Big) d\tau\\
&\leq C e^{-C_{*} t}\|\nabla^2 X(T_{*})\|^2+C(1+t)^{-\frac{7}{2}}\Big(\|X(0)\|_{L^1}^2+1\Big)+CN^2(t)(1+t)^{-\frac{17}{4}}\\
&\leq C e^{-C_{*} t}\|\nabla^2 X(T_{*})\|^2+C(1+t)^{-\frac{7}{2}}\Big(\|X(0)\|_{L^1}^2+1+N^2(t)(1+t)^{-\frac{3}{4}}\Big).
\end{aligned}
\end{eqnarray}
In addition, according to (\ref{A.5}) and lemma \ref{LA.1}, it gets for $k\in[0,2]$
\begin{eqnarray}\label{L2.8.13}
\|\nabla^k X(t)\|^2\leq C\Big(\|\nabla^k X^L (t)\|^2+\|\nabla^k X^h (t)\|^2 \Big)\leq C\Big(\|\nabla^k X^L (t)\|^2+\|\nabla^2 X(t)\|^2 \Big).
\end{eqnarray}
Substituting (\ref{L2.8.11}) and (\ref{L2.8.12}) into (\ref{L2.8.13}), for $k\in[0,2]$, $t\geq T_{*}$, we have
\begin{eqnarray}\label{L2.8.14}
\begin{aligned}
\|\nabla^k X(t)\|^2 &\leq C (1+t)^{-\frac{3}{2}-k} \Big(\|X(0)\|_{L^1}^2+1+N^2(t)(1+t)^{-\frac{3}{4}}\Big)\\
&+ C e^{-C_{*} t}\|\nabla^2 X(T_{*})\|^2+C(1+t)^{-\frac{7}{2}}\Big(\|X(0)\|_{L^1}^2+1+N^2(t)(1+t)^{-\frac{3}{4}}\Big)\\
&\leq C (1+t)^{-\frac{3}{2}-k}\Big(\|X(0)\|_{L^1}^2+1+N^2(t)(1+t)^{-\frac{3}{4}}\Big)+ C e^{-C_{*} t}\|\nabla^2 X(T_{*})\|^2.
\end{aligned}
\end{eqnarray}
Thus, from (\ref{L2.8.1}) and (\ref{L2.8.14}), $\exists~C_{**}>0$ such that
\begin{eqnarray}\nonumber
N^2(t)\leq C_{**}\Big(\|X(0)\|_{L^1}^2+1+N^2(t)(1+t)^{-\frac{3}{4}}+\|\nabla^2 X(T_{*})\|^2\Big).
\end{eqnarray}
Further, $\exists~\widetilde{T}>0$ for $t\geq \widetilde{T}$, it holds
\begin{eqnarray}\nonumber
C_{**}(1+t)^{-\frac{3}{4}}\leq\frac{1}{2},
\end{eqnarray}
such that
\begin{eqnarray}\nonumber
N^2(t)\leq 2C_{**}\Big(\|X(0)\|_{L^1}^2+1+\|\nabla^2 X(T_{*})\|^2\Big),
\end{eqnarray}
combining it with (\ref{Th2}), we can obtain for $\forall t\geq \widetilde{T}$
\begin{eqnarray}\nonumber
N(t)\leq C.
\end{eqnarray}
According to (\ref{L2.8.2}), we have proved (\ref{L2.8}), further more, we have confirmed Theorem \ref{Theorem1.2}.
\end{proof}
\section*{Appendix}
\appendix
\section{Frequency decomposition}
For $\chi_{x}=\frac{1}{i}\nabla=\frac{1}{i}(\partial_{x_{1}},\partial_{x_{2}},\partial_{x_{3}})$, $\phi_{0}(\chi_{x})$ and $\phi_{1}(\chi_{x})$ are pseudo differential operators.  $0\leq \phi_{0}(\xi),\phi_{1}(\xi)\leq 1$ $(\xi\in R^3)$ are cut off and also smooth functions which satisfy
\begin{eqnarray}\nonumber
 \phi_{0}(\xi)=
 \begin{cases}
 0,~~|\xi|>b_{0},\\
 1,~~|\xi|<\frac{b_{0}}{2},
 \end{cases}~~
 \phi_{1}(\xi)=
 \begin{cases}
 1,~~|\xi|>B_{0}+1,\\
 0,~~|\xi|<B_{0},
 \end{cases}
\end{eqnarray}
Here, $b_{0}$ and $B_{0}$ (fixed constants) satisfy
\begin{eqnarray}\label{A.2}
0<b_{0}\leq \sqrt{\frac{P'(1)}{\eta+\mu}},
\end{eqnarray}
\begin{eqnarray}\label{A.3}
B_{0} \geq \max\Big\{\sqrt{\frac{48\delta_{0}}{\mu}},\sqrt{\frac{48C_{1}\delta_{0}}{\mu}},\sqrt{\frac{6P'(1)\delta_{0}}{\mu}},\sqrt{\frac{32C_{1}\delta_{0}}{\nu}},\sqrt{\frac{4P'(1)\delta_{0}}{\nu}}\Big\},
\end{eqnarray}
Thus, by Fourier transform, for a function $g(x)\in L^2(R^3)$, let's define a frequency decomposition $(g^l(x),g^m(x),g^h(x))$
\begin{eqnarray}\label{A.1}
g^l(x)=\phi_{0}(\chi_{x})g(x),~~
g^m(x)=(I-\phi_{0}(\chi_{x})-\phi_{1}(\chi_{x}))g(x),~~
g^h(x)=\phi_{1}(\chi_{x})g(x).
\end{eqnarray}
Further, denote
\begin{eqnarray}\label{A.4}
g^{L}(x):=g^l(x)+g^m(x),~~g^{H}(x):=g^m(x)+g^h(x),
\end{eqnarray}
one has
\begin{eqnarray}\label{A.5}
g(x)=g^l(x)+g^m(x)+g^h(x):=g^l(x)+g^{H}(x):=g^{L}(x)+g^h(x).
\end{eqnarray}
Using Plancherel theorem and (\ref{A.1}), the following inequalities can be gained.
\begin{Lemma}\label{LA.1}
$\forall~\text{integers}~p_{0}, p, p_{1}, s$, for $p_{0}\leq p\leq p_{1} \leq s$ and $g(x)\in H^s(R^3)$, it gets
\begin{eqnarray}\label{A.6}
\|\nabla^p g^l\|\leq b_{0}^{p-p_{0}}\|\nabla^{p_{0}} g^l\|,~~~\|\nabla^p g^l\|\leq \|\nabla^{p_{1}} g\|,
\end{eqnarray}
\begin{eqnarray}\label{A.7}
\|\nabla^p g^h\|\leq \frac{1}{B_{0}^{p_{1}-p}}\|\nabla^{p_{1}} g^h\|,~~~\|\nabla^p g^h\|\leq \|\nabla^{p_{1}} g\|,
\end{eqnarray}
and
\begin{eqnarray}\label{A.8}
b_{0}^p\|g^m\|\leq \|\nabla^p g^m\|\leq B_{0}^p\|g^m\|.
\end{eqnarray}
\end{Lemma}
\section{The analysis of linearized equations}
Denote $\Omega :=(-\Delta)^{\frac{1}{2}}$, and $\Upsilon:=\Omega^{-1}  \diverg u$, then, $u=-\Omega^{-1} \nabla \Upsilon-\Omega^{-1} \diverg (\Omega^{-1} \curl u)$.\\
From (\ref{2.1}), it can get
\begin{eqnarray}\label{B.1}
 \begin{cases}
   \partial_{t}q+\Omega \Upsilon=\mathfrak{a},\\
   \partial_{t}\Upsilon -P'(1)\Omega q-(\eta+\mu)\Delta \Upsilon = \Omega^{-1} \diverg \mathfrak{b},\\
   \partial_{t}M-\nu\Delta M=\mathfrak{c},
 \end{cases}
\end{eqnarray}
while $\Gamma u=\Omega^{-1} \curl u$ satisfies
\begin{eqnarray}\label{B.2}
 \begin{cases}
   \partial_{t}\Gamma u -\mu \Delta \Gamma u = \Gamma \mathfrak{b},\\
   \Gamma u(x,0)=\Gamma u_{0}(x).
 \end{cases}
\end{eqnarray}
 Thus, based on the above analysis, the analysis of $u$ only relies on the analysis of $\Upsilon$ and $\Gamma u$.\\
By the application of Fourier transform on (\ref{B.1}), we can check
\begin{eqnarray}\label{B.3}
 \begin{cases}
    \partial_{t}\widehat{q}+|\xi| \widehat{\Upsilon}=\widehat{\mathfrak{a}},\\
   \partial_{t}\widehat{\Upsilon} -P'(1)|\xi| \widehat{q}+(\eta+\mu) |\xi|^2 \widehat{ \Upsilon} = \widehat{\Omega^{-1}\diverg \mathfrak{b}},\\
   \partial_{t}\widehat{M}+\nu|\xi|^2 \widehat{M}=\widehat{\mathfrak{c}},
 \end{cases}
\end{eqnarray}
further, the linearized equations are as below
\begin{eqnarray}\label{B.4}
 \begin{cases}
    \partial_{t}\widehat{q}+|\xi| \widehat{\Upsilon}=0,\\
   \partial_{t}\widehat{\Upsilon} -P'(1)|\xi| \widehat{q}+(\eta+\mu) |\xi|^2 \widehat{\Upsilon} = 0,\\
   \partial_{t}\widehat{M}+\nu|\xi|^2 \widehat{M}=0.
 \end{cases}
\end{eqnarray}
In fact, (\ref{B.4}) is equivalent to
\begin{eqnarray}\label{B.5}
 \frac{d}{dt}\begin{pmatrix}
\widehat{q}\\
\widehat{\Upsilon}\\
\widehat{M}
\end{pmatrix}
+\mathbb{Q}\begin{pmatrix}
\widehat{q}\\
\widehat{\Upsilon}\\
\widehat{M}
\end{pmatrix}=0,
\end{eqnarray}
while
\begin{eqnarray}\nonumber
\mathbb{Q}=
\begin{pmatrix}
0& |\xi| &0\\
-P'(1)|\xi| &(\eta+\mu) |\xi|^2 &0\\
0&0&\nu|\xi|^2
\end{pmatrix}.
\end{eqnarray}
\subsection{The study of low-frequency portion}
By (\ref{B.4}), one may calculate
\begin{eqnarray}\label{B.6}
 \frac{d}{dt}\Big(\frac{P'(1)|\widehat{q}|^2+|\widehat{\Upsilon}|^2+|\widehat{M}|^2}{2}\Big)+(\eta+\mu) |\xi|^2 |\widehat{\Upsilon}|^2+\nu|\xi|^2 |\widehat{M}|^2=0.
\end{eqnarray}
Adding up $(\ref{B.4})_{1} \times\bar{\widehat{\Upsilon}}$ with $\overline{(\ref{B.4})_{2}}\times\widehat{q}$, it can check
\begin{eqnarray}\label{B.7}
 \frac{d}{dt}Re(\widehat{q}\bar{\widehat{\Upsilon}})+|\xi||\widehat{\Upsilon}|^2-P'(1)|\xi||\widehat{q}|^2 =-(\eta+\mu)|\xi|^2 Re(\widehat{q}\bar{\widehat{\Upsilon}}).
\end{eqnarray}
Choosing $\delta_{*}>0$ (fixed small constant ) and summing up $-\delta_{*}|\xi|\times$ (\ref{B.7}) with (\ref{B.6}), yields directly
\begin{eqnarray}\label{B.8}
\begin{aligned}
  &\frac{d}{dt}\Big(\frac{P'(1)|\widehat{q}|^2+|\widehat{\Upsilon}|^2+|\widehat{M}|^2-2\delta_{*}|\xi|Re(\widehat{q}\bar{\widehat{\Upsilon}})}{2}\Big)\\
  &+(\eta+\mu) |\xi|^2 |\widehat{\Upsilon}|^2+\nu|\xi|^2 |\widehat{M}|^2-\delta_{*}|\xi|^2|\widehat{\Upsilon}|^2+P'(1)\delta_{*}|\xi|^2|\widehat{q}|^2 \\ &=\delta_{*}(\eta+\mu)|\xi|^3 Re(\widehat{q}\bar{\widehat{\Upsilon}})\\
  &\leq \frac{\delta_{*}(\eta+\mu)^2}{2P'(1)}|\xi|^4 |\widehat{\Upsilon}|^2+\frac{P'(1)\delta_{*}}{2}|\xi|^2|\widehat{q}|^2.
\end{aligned}
\end{eqnarray}
By noticing
\begin{eqnarray}\nonumber
0<\delta_{*}\leq \min\Big\{\frac{1}{2},\frac{\eta+\mu}{2}\Big\},
\end{eqnarray}
it can get from (\ref{B.8})
\begin{eqnarray}\label{B.9}
\begin{aligned}
  &\frac{d}{dt}\Big(\frac{P'(1)|\widehat{q}|^2+|\widehat{\Upsilon}|^2+|\widehat{M}|^2-2\delta_{*}|\xi|Re(\widehat{q}\bar{\widehat{\Upsilon}})}{2}\Big)+\frac{\eta+\mu}{2} |\xi|^2 |\widehat{\Upsilon}|^2+\nu|\xi|^2 |\widehat{M}|^2+\frac{\delta_{*}P'(1)}{2}|\xi|^2|\widehat{q}|^2 \\
  &\leq \frac{(\eta+\mu)^2}{4P'(1)}|\xi|^4 |\widehat{\Upsilon}|^2,
\end{aligned}
\end{eqnarray}
further, noticing a small constant $b_{0}$, for
\begin{eqnarray}\nonumber
|\xi|\leq b_{0}\leq \sqrt{\frac{P'(1)}{\eta+\mu}},
\end{eqnarray}
it can obtain from (\ref{B.9})
\begin{eqnarray}\label{B.10}
\begin{aligned}
  \frac{d}{dt}\ell(\xi,t)+\frac{\eta+\mu}{4} |\xi|^2 |\widehat{\Upsilon}|^2+\nu|\xi|^2 |\widehat{M}|^2+\frac{\delta_{*}P'(1)}{2}|\xi|^2|\widehat{q}|^2 \leq 0,
\end{aligned}
\end{eqnarray}
where
\begin{eqnarray}\nonumber
\ell(\xi,t):=\frac{P'(1)|\widehat{q}|^2+|\widehat{\Upsilon}|^2+|\widehat{M}|^2-2\delta_{*}|\xi|Re(\widehat{q}\bar{\widehat{\Upsilon}})}{2}.
\end{eqnarray}
Because of $\delta_{*}b_{0}\leq\min\{\frac{P'(1)}{2},\frac{1}{2}\}$, one has
\begin{eqnarray}\nonumber
\ell(\xi,t)\sim|\widehat{q}|^2+|\widehat{\Upsilon}|^2+|\widehat{M}|^2,
\end{eqnarray}
so, for $|\xi|\leq b_{0}$, $\exists~\text{constant}~C_{l}>0$ such that
\begin{eqnarray}\label{B.11}
C_{l}|\xi|^2 \ell(\xi,t)\leq \frac{\eta+\mu}{4} |\xi|^2 |\widehat{\Upsilon}|^2+\nu|\xi|^2 |\widehat{M}|^2+\frac{\delta_{*}P'(1)}{2}|\xi|^2|\widehat{q}|^2.
\end{eqnarray}
Thus, from (\ref{B.10}) and (\ref{B.11}), it obtains
\begin{eqnarray}\label{B.12}
\ell(\xi,t)\leq  e^{-C_{l}|\xi|^2 t}\ell(\xi,0),~|\xi|\leq b_{0}.
\end{eqnarray}
\subsection{The study of middle-frequency portion}

Let's calculate the characteristic polynomial of $\mathbb{Q}$
\begin{eqnarray}\nonumber
\begin{aligned}
\mathbb{Q}_{\lambda_{0}}=|Q-\lambda_{0}I|
=\Big(\lambda_{0}^2-(\eta+\mu)|\xi|^2 \lambda_{0} +P'(1)|\xi|^2 \Big)\Big(\nu|\xi|^2-\lambda_{0}\Big)=-\alpha_{0}\lambda_{0}^3+\alpha_{1}\lambda_{0}^2-\alpha_{2}\lambda_{0}+\alpha_{3},
\end{aligned}
\end{eqnarray}
here
\begin{eqnarray}\nonumber
\begin{aligned}
\alpha_{0}=1,~~ \alpha_{1}=(\eta+\mu)|\xi|^2+\nu|\xi|^2,~~ \alpha_{2}=\nu(\eta+\mu)|\xi|^4+P'(1)|\xi|^2,~~ \alpha_{3}=\nu P'(1)|\xi|^4.
\end{aligned}
\end{eqnarray}
Since $\alpha_{0}$-$\alpha_{4}$ all positive, by Routh-Hurwitz theorem, all roots of $\mathbb{Q}_{\lambda_{0}}$ have positive real part if and only if $\mathbb{Q}_{1}>0$ and $\mathbb{Q}_{2}>0$
\begin{eqnarray}\nonumber
\begin{aligned}
&\mathbb{Q}_{1}=\alpha_{1}=(\eta+\mu)|\xi|^2+\nu|\xi|^2>0,\\
&\mathbb{Q}_{2}=
\begin{vmatrix}
\alpha_{1}&\alpha_{0}\\
\alpha_{3}&\alpha_{2}
\end{vmatrix}
=\nu(\eta+\mu)^2|\xi|^6+(\eta+\mu)P'(1)|\xi|^4+\nu^2(\eta+\mu)|\xi|^6>0.
\end{aligned}
\end{eqnarray}
According to the analysis in subsection 3.3 proofed in \cite{Danchin}, Lemma \ref{LB1} is attainable.
\begin{Lemma}\label{LB1}
For $\forall~\text{constant}~\mathfrak{b}~\text{and}~\mathfrak{B}~(0<\mathfrak{b}<\mathfrak{B})$, $\exists$ constant $\varsigma>0$ which relies only on $\eta$, $\nu$, $\mathfrak{b}$ and $\mathfrak{B}$ such that for $\forall~\mathfrak{b}\leq |\xi|\leq \mathfrak{B}$ and $t\in R^{+}$
\begin{eqnarray}\label{B.13}
\begin{aligned}
|e^{-t\mathbb{Q}}|\leq Ce^{-\varsigma t}.
\end{aligned}
\end{eqnarray}
\end{Lemma}
From (\ref{B.5}) and (\ref{B.13}), for $\forall~\mathfrak{b}\leq |\xi|\leq \mathfrak{B}$, one has
\begin{eqnarray}\label{B.14}
\begin{aligned}
|(\widehat{q},\widehat{\Upsilon},\widehat{M})(\xi,t)|=|e^{-t\mathbb{Q}}(\widehat{q},\widehat{\Upsilon},\widehat{M})(\xi,0)|\leq Ce^{-\varsigma t}|(\widehat{q},\widehat{\Upsilon},\widehat{M})(\xi,0)|.
\end{aligned}
\end{eqnarray}
\subsection{Estimation of $\widehat{\Gamma u}(\xi,t)$}
Taking Fourier transform on (\ref{B.2}), the linearized portion of that is as below
\begin{eqnarray}\label{B.15}
   \partial_{t}\widehat{\Gamma u} +\mu |\xi|^2\widehat{\Gamma u} = 0,
\end{eqnarray}
for $\forall~|\xi|\geq 0$, one has
\begin{eqnarray}\label{B.16}
   |\widehat{\Gamma u}(\xi,t)|^2\leq C e^{-\mu |\xi|^2 t}|\widehat{\Gamma u}(\xi,0)|^2.
\end{eqnarray}

\vskip2mm
\renewcommand\refname{References}

\end{document}